\documentclass[reqno,12pt]{amsart}

\numberwithin{equation}{section}

\usepackage{amsmath}

\usepackage{esint} 

\usepackage{amsthm}

\usepackage{calligra}

\usepackage{epsfig}

\usepackage{psfrag}

\usepackage{graphicx}

\usepackage{graphpap,latexsym,epsf}

\usepackage{color}

\usepackage{amssymb,mathrsfs,enumerate}

\usepackage{endnotes}

\usepackage{calligra}

\usepackage[colorinlistoftodos]{todonotes} 

\definecolor{citegreen}{rgb}{0,0.6,0}

\definecolor{refred}{rgb}{0.8,0,0}

\usepackage[colorlinks, citecolor=citegreen, linkcolor=refred]{hyperref}

\usepackage{enumitem}

\newcommand{\R}{\mathbb{R}}

\newcommand{\Z}{\mathbb{Z}}

\newcommand{\SSS}{\mathbb{S}}

\def\HHH{{\rm H}}

\def\RRR{{\mathrm R}}

\newcommand{\pa}{\partial}

\newcommand{\Om}{\Omega}

\newcommand{\ep}{\varepsilon}

\newcommand{\Ric}{{\rm Ric}}

\newcommand{\Ro}{{\rm R}}

\newcommand{\na}{\nabla}

\mathchardef\emptyset="001F

\definecolor{vgreen}{rgb}{0.1,0.5,0.2}

\definecolor{viola}{RGB}{85,26,139}

\newtheorem{theorem}{Theorem}[section]

\newtheorem{corollary}[theorem]{Corollary}

\newtheorem{remark}[theorem]{Remark}

\newtheorem{lemma}[theorem]{Lemma}

\definecolor{byzantium}{rgb}{0.44, 0.16, 0.39}

\definecolor{amber}{rgb}{1.0, 0.75, 0.0}

\definecolor{darkmagenta}{rgb}{0.55, 0.0, 0.55}

\definecolor{fuzzywuzzy}{rgb}{0.8, 0.4, 0.4}

\definecolor{brown}{rgb}{0.2, 0.08, 0.08}

\definecolor{arancio}{rgb}{1.0, 0.13, 0.0}

\begin{document}
\title{{A Green's function proof of the Positive Mass Theorem}}
\author[V.~Agostiniani]{Virginia Agostiniani}
\address{V.~Agostiniani, Universit\`a degli Studi di Trento,
via Sommarive 14, 38123 Povo (TN), Italy}
\email{virginia.agostiniani@unitn.it}
\author[L.~Mazzieri]{Lorenzo Mazzieri}
\address{L.~Mazzieri, Universit\`a degli Studi di Trento,
via Sommarive 14, 38123 Povo (TN), Italy}
\email{lorenzo.mazzieri@unitn.it}
\author[F.~Oronzio]{Francesca Oronzio}
\address{F.~Oronzio, Universit\`a degli Studi di Roma ``La Sapienza”, Piazzale Aldo Moro 5, 00185 Roma (RO), Italy}
\email{francesca.oronzio@uniroma1.it}
\maketitle

\begin{abstract}
In this paper, a new proof of the Positive Mass Theorem is established through a newly discovered monotonicity formula, holding along the level sets of the Green's function of an asymptotically flat $3$-manifold. In the same context and for $1<p<3$, a Geroch-type calculation is performed along the level sets of $p$-harmonic functions, leading to a new proof of the Riemannian Penrose Inequality {under favourable assumptions}. A new characterisation of scalar curvature lower bounds in terms of the monotonicity formulas is also given. 
\end{abstract}

\bigskip





\section{A monotonic quantity for the Green's functions}\label{effmon}


In the context of contemporary geometric analysis, monotonicity formulas are known to play a central role not only because of their tremendous implications, but also because they clarify the theory, eventually leading to a deeper understanding of the overall picture. Some of the most relevant examples come from the study of geometric flows such as, e.g., Huisken's monotonicity formulas for the Mean Curvature Flow~\cite{Hui_1990}, Perelman's entropy formulas for the Ricci Flow~\cite{PERELMAN}, or Geroch's monotonicity of the Hawking Mass along the Inverse Mean Curvature Flow~\cite{Ger,Hui_Ilm_2001}. However, this is not the only scene where these formulas play a dominant role. For example, it must be noticed that a monotonicity formula is also at the core of comparison geometry, through the nowadays classical Bishop-Gromov Volume Comparison Theorem. Starting from this observation, Colding~\cite{Colding_Acta} and Colding-Minicozzi~\cite{Col_Min_2,Col_Min_3} discovered in recent years analogous monotonic quantities on manifolds with nonnegative Ricci curvature, where the level set flow of the distance function is replaced by the level set flow of a harmonic function. This new class of monotonicity formulas, together with their extension to the case of $p$-harmonic functions, $1<p<n$, has revealed to be extremely flexible and powerful, leading to the proof of new geometric inequalities~\cite{Ago_Fog_Maz_1,Ben_Fog_Maz_1,Xia_Yin} as well as to a new proof of classical results~\cite{Ago_Maz_CV,Ago_Fog_Maz_2}. More in general, level set methods for harmonic functions have been recently employed to investigate the geometry of asymptotically flat initial data in general relativity~\cite{Ago_Maz_CMP,Ago_Maz_Oro_1,Bra_Kaz_Khu_Ste_2019,Hir_Kaz_Khu}. Further comments on the latter point will be given in the next section. Here, we establish the following monotonicity result, holding along the level set flow of the Green's functions on complete nonparabolic $3$-manifolds with nonnegative scalar curvature, whose topology is sufficiently simple. 

\begin{theorem}\label{effectivemonotonicity}
Let $(M,g)$ be a complete, noncompact, nonparabolic $3$-dimensional Riemannian manifold with nonnegative scalar curvature and $H_2(M;\Z)=\{0\}$. Let $u$ be the maximal distributional solution to 
\begin{equation}\label{f0}
\begin{cases}
\Delta u=4\pi\delta_{o}\  &\mathrm{in} \  M,\\
\quad u \to 1 &\mathrm{at} \  \infty ,
\end{cases}
\end{equation}
for some $o \in M$, and let $F:(0, + \infty)\to \R$ be the function defined as 
\begin{equation}\label{eq0}
F(t)\,\, := \,\, 4\pi t \,\,-\,\,\, t^{2} \!\!\!\!\!\!\!\!\!\!\!\!\int\limits_{\{u=1-\frac{1}{t}\}\setminus\mathrm{Crit}(u)}\!\!\!\!\!\!\!\!\!\!\!\!  \vert \nabla u \vert\, \mathrm{H} \,\, d\sigma \, \, + \,\, t^{3} \!\!\!\!\!\!\!\!\!\!\!\! \int\limits_{\{u=1-\frac{1}{t}\}\setminus\mathrm{Crit}(u)}\!\!\!\!\!\!\!\!\!\!\!\! \vert \nabla u \vert^{2} \,\, d\sigma \,, 
\end{equation}
where $\mathrm{Crit}(u)$ is the set of the critical points of $u$,
and
$\mathrm{H}$ is the mean curvature of the level set $\{u= 1 - 1/t\}\setminus\mathrm{Crit}(u)$ computed with respect to the $\infty$--pointing unit normal vector field $\nu={\nabla u}/{\vert \nabla u \vert}$. Then, we have that 
\begin{equation*}
0 < s \leq t < +\infty \quad \Rightarrow \quad  F(s) \, \leq  \, F(t) \, ,
\end{equation*}
provided $1- 1/s$ and $1 - 1/t$ are regular values of $u$. 
\end{theorem}
More specifically, we are going to prove that, under the above assumptions, the function $F$ admits a nondecreasing locally absolutely continuous representative, defined in $(0,+\infty)$ and still denoted by $F$, whose $L^1_{loc}$-weak derivative satisfies
\begin{align}\label{F'}
F'(t) \, = \, 4\pi +\!\!\!\!\!\!\!\!\!\!\!\!\!\!\!\int\limits_{\{u=1-\frac{1}{t}\}\setminus\mathrm{Crit}(u)}\!\!\!\!\!\!\!\!\!\bigg[-\frac{\rm{R}^{\Sigma_{t}}}{2} + \frac{\vert\,\nabla^{\Sigma_t}\vert \nabla u\vert\,\vert^{2}}{\vert \nabla u \vert^{2}}+\frac{\Ro}{2}+\frac{\,\vert \mathring{\mathrm{h}}\vert^{2}}{2} +\frac{3}{4}\left(\,\frac{2\vert \nabla u \vert}{1-u} -\mathrm{H}\right)^{\!\!2}\,\bigg] d\sigma
\end{align}
a.e. in $(0,+\infty)$. Here, $\rm{R}^{\Sigma_{t}}$ and $\nabla^{\Sigma_{t}}$ denote the scalar curvature and the Levi-Civita connection of the metric induced by $g$ {on the regular part $\Sigma_t$ of the level set $\{ u= 1- 1/t\}$}. By $\mathrm{h}$ and $\mathrm{H}$, we indicate the second fundamental form and the mean curvature {of $\Sigma_t$}, both computed with respect to the $\infty$--pointing unit normal vector field $\nu={\nabla u}/{\vert \nabla u \vert}$, wheras $\mathring{\mathrm{h}}$ {represents the traceless second fundamental form of $\Sigma_t$.} In particular, as soon as $\Sigma_t$ is a connected regular level set of $u$, one can easily deduce that $F'(t)$ is nonnegative. In fact
$$
4 \pi \,- \! \int\limits_{\Sigma_t} \frac{\rm{R}^{\Sigma_{t}}}{2} \, d\sigma = 4 \pi \, - \! 2\pi \chi(\Sigma_t)  \geq 0 \, 
$$
by the Gauss-Bonnet Theorem, and all the remaining terms in~\eqref{F'} are manifestly nonnegative.

\begin{remark}
It should be noted that, in a recent work~\cite{Mun_Wan}, a new monotonicity formula holding along the level sets of the Green's function was discovered by Munteanu and Wang, working essentially in the same framework as in Theorem~\ref{effectivemonotonicity}. However, both the monotonicity formulas and the geometric conclusions are considerably different from ours.
\end{remark}

Observe that the existence {of the solution $u$ to~\eqref{f0} is equivalent to the existence of the} minimal positive Green's function $\mathcal{G}_{o}$ vanishing at infinity and having a pole at $o$, and it is guaranteed in many reasonable frameworks, such as the one -- particularly relevant to us -- of asymptotically flat manifolds. Indeed, $u$ and $\mathcal{G}_{o}$ are related by the identity $u=1-4\pi\mathcal{G}_{o}$. As it is well known, $u$ is smooth on $M \setminus \{ o\}$ and proper, so that its  level sets are compact. It follows then from~\cite[Theorem~1.7]{Hardt1} that they also have finite $2$-dimensional Hausdorff measure. For the reader's convenience, we also recall that the set $\mathrm{Crit}(u)$ has locally finite $1$-dimensional Hausdorff measure (see for instance~\cite[Theorem~1.1]{Hardt2}), and that the set of the critical values of $u$ has zero Lebesgue measure by Sard's Theorem, whereas the set of regular values of $u$ is open, by the same argument as in~\cite[Theorem~2.3]{Ago_Maz_Oro_1}. Building on the previous observations, it is not hard to realise that the function $F$ given in~\eqref{eq0} is well defined, as its summands come from the integration
of bounded functions on sets with finite measure. To justify this latter sentence, one only needs to check that $|\, |\nabla u|  \, \mathrm{H} \, |$ is bounded on $\{u=\tau\}\setminus\mathrm{Crit}(u)$, for every $\tau \in (-\infty , 1)$. Since $u$ is harmonic, $\mathrm{H}$ can be expressed as 
\begin{equation}\label{H}
\mathrm{H} \, = \, - \, \frac{\nabla \nabla u (\nabla u,\nabla u)}{\vert \nabla u\vert^{3}}\,=\,-\,\frac{\langle \nabla \vert \nabla u\vert,\!\nabla u\rangle}{ \vert \nabla u\vert^{2}},
\end{equation}
away from $\mathrm{Crit}(u)$, where angle brackets $\langle \cdot\,,\cdot\rangle$ denote the scalar product, taken with respect to the metric $g$. Consequently, one has that
\begin{equation}
\label{eq:hdubound}
\big\vert \,\vert \nabla u \vert \mathrm{H}\,\big\vert\leq \vert\nabla \nabla u (\nu,\nu) \vert\leq \vert \nabla \nabla u \vert  \, ,\quad 
\end{equation}
whenever $\vert \nabla u\vert \neq 0$. We are now ready to prove Theorem~\ref{effectivemonotonicity}. For the sake of exposition, we first give a proof under the favourable assumption that $\mathrm{Crit}(u) = \emptyset$, and then we proceed with the proof of the full statement.

\begin{proof}[Proof in the absence of critical points.]
In this case, all of the level sets of $u$ are regular and diffeomorphic to the 2-sphere, and in turn the function $F$ is everywhere continuously differentiable in its domain of definition. We claim that $F'(t) \geq 0$ for every $t \in (0, + \infty)$. We start observing that
\begin{align}
\frac{d}{dt} \!\!\!\! \int\limits_{\{u=1-\frac{1}{t}\}} \!\!\!\!\vert \nabla u \vert^{2} \, d\sigma \, &  = \, -  \frac{1}{t^{2}} \!\!\!\! \int\limits_{\{u=1-\frac{1}{t}\}} \!\!\!\! \vert \nabla u \vert\,\mathrm{H}\,  d\sigma\, ,\nonumber \\
\frac{d}{dt}\!\!\!\! \int\limits_{\{u=1-\frac{1}{t}\}} \!\!\!\!\vert \nabla u \vert\, \mathrm{H} \, d\sigma & \, = \, -\frac{1}{t^{2}}\!\!\!\! \int\limits_{\{u=1-\frac{1}{t}\}}  \!\!\!\! \vert \nabla u \vert  \left[ \, \Delta_{\Sigma_{t}} \!\left(\frac{1}{\vert \nabla u \vert}\right) + \, \frac{\vert \mathrm{h}\vert^{2}+\Ric (\nu,\nu)}{\vert \nabla u\vert} \, \right]d\sigma\,,\label{eq20}
\end{align}
where $\Delta_{\Sigma_{t}}$ is the Laplace-Beltrami operator of the metric induced {on $\Sigma_{t}=\{u=1-\frac{1}{t}\}$.} 
With the help of the Gauss equation, the integrand on the right hand side of~\eqref{eq20} can be expressed as 
\begin{align*}
\vert \nabla u \vert & \left[ \, \Delta_{\Sigma_{t}} \!\left(\frac{1}{\vert \nabla u \vert}\right) + \, \frac{\vert \mathrm{h}\vert^{2}+\Ric (\nu,\nu)}{\vert \nabla u\vert} \, \right]
= &\\
 & \qquad \qquad \qquad \qquad = -\, \Delta_{\Sigma_{t}}(\log \vert \nabla u \vert ) +\frac{\vert\,\nabla^{\Sigma_t}\vert \nabla u\vert\,\vert^{2}}{\vert \nabla u \vert^{2}} +\frac{\Ro}{2} -\!\frac{\,\rm{R}^{\Sigma_{t}}}{2} {+ \frac{\,\vert \mathring{\mathrm{h}}\vert^{2}}{2}   +\frac{3}{4}\mathrm{H}^{2}.}\nonumber
\end{align*}
Substituting the latter expression into~\eqref{eq20} and using standard manipulations, one easily arrives at 
\begin{equation}
\label{monoliscia}
F'(t) \, = \, 4\pi - \int\limits_{\Sigma_t}\!\frac{\,\rm{R}^{\Sigma_{t}}}{2} \, d\sigma  \, + \,  \int\limits_{\Sigma_t}\left[ \, \frac{\vert\,\nabla^{\Sigma_t}\vert \nabla u\vert\,\vert^{2}}{\vert \nabla u \vert^{2}}+\frac{\Ro}{2}+\frac{\,\vert \mathring{\mathrm{h}}\vert^{2}}{2}\, +\frac{3}{4}\left(\,\frac{2\vert \nabla u \vert}{1-u} -\mathrm{H}\right)^{2}\,\right] d\sigma.
\end{equation}
Now, we notice that the last summand of the right hand side is always nonnegative, as the scalar curvature of $(M,g)$ is nonnegative by assumption. The first two summands also give a nonnegative contribution, by virtue of Gauss-Bonnet Theorem combined with the observation that $\Sigma_t$ is closed and connected for all $t \in (0,+\infty)$. It is worth pointing out that when critical points are not present, the connectedness of the level sets of $u$ follows by rather elementary considerations, without any further assumption on the topology of $(M,g)$. In fact, on one hand, the asymptotic behaviour of $u$ near the pole implies that in this region at least one level set of $u$ is necessarily diffeomorphic to a 2-sphere. On the other hand, the same must be true for every level set of $u$, since $\na u$ never vanishes. 
\end{proof}

{\begin{proof}
Let us consider the vector field $Y$, defined as 
\begin{align}\label{Y}
Y:=\frac{\nabla \vert \nabla u\vert}{(1-u)^{2}} + \frac{\vert \nabla u \vert}{(1-u)^{3}}\nabla u\,,
\end{align}
where $u$ is the solution to problem~\eqref{f0}. The vector field $Y$ is well defined and smooth on the open set $M_{o}\setminus \mathrm{Crit}(u)$, where $M_{o}$ is defined as
$$M_{o}:=M\setminus\{o\} \, .
$$ 
With the help of the Bochner formula, the divergence of $Y$ on $M_{o}\setminus \mathrm{Crit}(u)$ can be expressed as 
\begin{align*}
\mathrm{div}(Y)=\frac{\vert \nabla u \vert}{(1-u)^{2}} \left[\frac{3\vert \nabla u \vert^{2}}{(1-u)^{2}}+\frac{3 \langle\nabla \vert \nabla u\vert,\! \nabla u  \rangle}{ (1-u) \, \vert \nabla u\vert}
+\frac{
{\vert \nabla \nabla u\vert^{2}}
-
{\vert\,\nabla\vert \nabla u\vert\,\vert^{2}}
+
{\Ric(\na u,\na u)}
}
{|\na u|^2}\right],
\end{align*}
where in the computation we took advantage of the fact that $u$ is harmonic. Using the Gauss equation in combination with the standard identity $\vert \nabla \nabla u\vert^{2}=\vert \nabla u\vert^{2} \vert \rm{h}\vert^{2}+\vert\,\nabla\vert \nabla u\vert\,\vert^{2}+\vert\,\nabla^{\Sigma_t}\vert \nabla u\vert\,\vert^{2}$, one can work out an equivalent expression for 
$\mathrm{div}(Y)$, adapted to {(the regular portion of) the level sets of $u$}, namely 
\begin{align}\label{div(Y)geom}
\mathrm{div}(Y)=\frac{\vert \nabla u \vert}{(1-u)^{2}}\,\left[ -\frac{\,\rm{R}^{\Sigma}}{2}+\frac{\vert\,\nabla^{\Sigma}\vert \nabla u\vert\,\vert^{2}}{\vert \nabla u \vert^{2}}+\frac{\Ro}{2}+\frac{\, \vert \mathring{\mathrm{h}}\vert^{2}}{2} +\frac{3}{4}\left( \frac{2\vert \nabla u \vert}{1-u} -\mathrm{H}\right)^{\!\!2} \, \right]\,.
\end{align}
Here, $\mathrm{h}, \mathrm{H}, \rm{R}^{\Sigma}$, and $\na^\Sigma$ are all referred {to the regular portion of} the level set of $u$ that passes through the point where $\mathrm{div}(Y)$ is computed. Setting 
\begin{equation*}
\Phi(t) \, := \, F(t) - 4\pi t \, = \,-\,\,\, t^{2} \!\!\!\!\!\!\!\!\!\!\!\!\int\limits_{\{u=1-\frac{1}{t}\}\setminus\mathrm{Crit}(u)}\!\!\!\!\!\!\!\!\!\!\!\!  \vert \nabla u \vert\, \mathrm{H} \,\, d\sigma \, \, + \,\, t^{3} \!\!\!\!\!\!\!\!\!\!\!\! \int\limits_{\{u=1-\frac{1}{t}\}\setminus\mathrm{Crit}(u)}\!\!\!\!\!\!\!\!\!\!\!\! \vert \nabla u \vert^{2} \,\, d\sigma  \,\, =  \!\!\!\!\!\!
 \!\!\!\!\!\!\int\limits_{\{u=1-\frac{1}{t}\}\setminus\mathrm{Crit}(u)}\!\!\!\!\!\!\!\!\!\!\!\!  \left\langle Y ,  \frac{\nabla u}{|\nabla u|} \right\rangle \, d\sigma 
\,, \end{equation*}
for every $t \in (0, +\infty)$, and using the Divergence Theorem, it is not hard to check that
\begin{align*}
\Phi(t)-\Phi(s) \,\, & = \!\!\!\!\!\!\!\!\!\!\int\limits_{\{1-\frac{1}{s}< u<1-\frac{1}{t}\}} \!\!\!\!\!\!\!\!\!\! \mathrm{div}(Y) \, d\mu \,\, = \\
&= \, \int\limits_{s}^{t} d\tau \!\!\!\int\limits_{\{u=1-\frac{1}{\tau}\} }\!\!\left[-\,\frac{\rm{R}^{\Sigma_{\tau}}}{2} + \frac{\vert\,\nabla^{\Sigma_\tau}\vert \nabla u\vert\,\vert^{2}}{\vert \nabla u \vert^{2}}+\frac{\Ro}{2}+\frac{\,\vert \mathring{\mathrm{h}}\vert^{2}}{2}\, +\frac{3}{4}\left(\,\frac{2\vert \nabla u \vert}{1-u} -\mathrm{H}\right)^{2}\,\right] d\sigma
\nonumber \, ,
\end{align*}
provided no critical values of $u$ are contained in the interval $[1-1/s, 1-1/t]$. {We aim at  proving that a similar statement holds true also in the presence of critical values, showing $\Phi\in W^{1,1}_{loc}(0,+\infty)$.} As a first step in this direction, we claim that 
\begin{equation}
\label{eq:claim}
\mathrm{div}(Y) \, \mathbb{I}_{M_{o}\setminus\mathrm{Crit}(u)}\in L^{1}_{loc}(M_o) \, ,
\end{equation}
where $\mathbb{I}_{M_{o}\setminus\mathrm{Crit}(u)}$ denotes the characteristic function of $M_{o}\setminus\mathrm{Crit}(u)$.
Let us observe that if $K$ is a compact subset of $M_o$, then, by Sard's Theorem, $K$ is contained in the set 
$$
E_{s}^{t}:=\left\{1-\frac{1}{s}<u<1-\frac{1}{t}\right\} \, ,
$$ 
for some $0<s<t$, such that both $1- 1/s$ and $1 - 1/t$ are regular values of $u$. As we are dealing with the case where critical values of $u$ are present in the interval $(1-1/s , 1-1/t)$, we have that necessarily the open subset $\{ 1-1/s < u < 1 - 1/t\}$ must contain critical points of $u$, so that the vector field $Y$ is no longer smooth and well defined everywhere. To overcome this issue, we 
consider a sequence of cut-off functions $\{ \eta_k\}_{k \in \mathbb{N}^+}$, where, for every $k \in \mathbb{N}^+$, the  function $\eta_k:[0,+\infty) \to [0,1]$ is smooth, nondecreasing, and obeying the following 
structural conditions:
\begin{align}
\eta_k (\tau) \equiv 0\quad \text{in $\left[0 \,,\frac{1}{2k}\,\right]$}\,,\qquad 
0\leq \eta_k'(\tau)\leq 2 k\quad \text{in $\left[\,\frac{1}{2k}\,,\frac{3}{2k}\,\right]$},
\qquad
\eta_k(\tau) \equiv 1\quad\text{in $\left[\,\frac{3}{2k} \, ,+\infty\!\right)$}\,.\,\nonumber
\end{align}
Using these cut-off functions, we define for every $k \in\mathbb{N}^+$ the vector field
\begin{equation}
Y_{k}\, := \, \, \eta_{{k}}\Big(\frac{\vert \nabla u\vert}{1-u}\Big)Y \, .\,\nonumber
\end{equation}
It is immediate to observe that for every $k \in\mathbb{N}^+$, $Y_k$ is a smooth vector field defined in $M_{o}$. Moreover, on any given compact subset of $M_{o}\setminus \mathrm{Crit}(u)$, the vector field $Y_k$ coincides with the vector field $Y$, provided $k$ is large enough.
For any such $Y_k$, the divergence is readily computed as follows
\begin{align}
\mathrm{div}(Y_{k})
& \, = \, \frac{\vert \nabla u \vert}{(1-u)^{2}} \left\{ \eta_{{k}}\Big(\frac{\vert \nabla u\vert}{1-u}\Big) \left[\frac{3\vert \nabla u \vert^{2}}{(1-u)^{2}}+\frac{\vert \nabla \nabla u\vert^{2}-\vert\,\nabla\vert \nabla u\vert\,\vert^{2}}{\vert \nabla u\vert^{2}}\right]    \right\}\nonumber\\
& \, + \,  \frac{\vert \nabla u \vert}{(1-u)^{2}}\left\{  \eta_{{k}}\Big(\frac{\vert \nabla u\vert}{1-u}\Big)\left[ \frac{ 3\langle\nabla \vert \nabla u\vert,\! \nabla u \rangle}{ (1-u) \vert \nabla u\vert}+ \frac{\Ric(\na u,\na u)}{|\na u|^2}\right] \,\,  \right\}\nonumber\\
& \, + \, \frac{|\na u|^2}{(1-u)^{3}}\,\, \eta'_{{k}}\Big(\frac{\vert \nabla u\vert}{1-u}\Big)\left\vert\,\frac{ \nabla u}{1-u}+ \frac{\nabla\vert \nabla u\vert}{|\na u|}\,\right\vert^{2}.\label{eq2}
\end{align}
An important feature of the above expression is that the last summand is always nonnegative. Taking into account these simple considerations and applying the  Divergence Theorem, one gets
\begin{align}
\label{tfci}
&\Phi(t)-\Phi(s) \,\,  = \!\!\!\!\!\!\!\!\!\!\int\limits_{\{1-\frac{1}{s}< u<1-\frac{1}{t}\}} \!\!\!\!\!\!\!\!\!\! \mathrm{div}(Y_k) \, d\mu \,\,\, \geq \!\!\!\!\!\!\!\!\! \int\limits_{\{1-\frac{1}{s}< u<1-\frac{1}{t}\}} \!\!\!\!\!\!\!\!\!\!\! P_k \, d\mu \,\,\,\,\, + \!\!\!\!\!\!\!\!\int\limits_{\{1-\frac{1}{s}< u<1-\frac{1}{t}\}}\!\!\!\!\!\!\!\!\!\! D_k \, d\mu\,,
\end{align}
where we set 
\begin{align}
P_k & \, := \, \eta_{{k}}\Big(\frac{\vert \nabla u\vert}{1-u}\Big)\,\,P\,,\quad \text{with}\quad P\,:=\,\frac{\vert \nabla u \vert}{(1-u)^{2}} \left[\frac{3\vert \nabla u \vert^{2}}{(1-u)^{2}}+\frac{\vert \nabla \nabla u\vert^{2}-\vert\,\nabla\vert \nabla u\vert\,\vert^{2}}{\vert \nabla u\vert^{2}}\right] \,,\nonumber\\
D_k & \, :=\,  \eta_{{k}}\Big(\frac{\vert \nabla u\vert}{1-u}\Big)\,\,D\,,\,\quad \text{with}\quad D\,:=\,\frac{\vert \nabla u \vert}{(1-u)^{2}}\left[ \frac{ 3\langle\nabla \vert \nabla u\vert,\! \nabla u \rangle}{ (1-u) \vert \nabla u\vert}+ \frac{\Ric(\na u,\na u)}{|\na u|^2}\right] \, .\nonumber
\end{align}
Concerning the functions $D_k$, we have that
$$
\vert D_k \vert
\,\, \leq \,\,
\frac{\vert \nabla u \vert}{(1-u)^{2}} \, \left[\, \frac{3\,\vert\, \nabla \na u\vert}{1-u}+\vert\Ric\vert \, \right]\, \mathbb{I}_{M_o\setminus\mathrm{Crit}(u)}\,.
$$
Since the function on the right hand side belongs to $L^{1}_{loc}(M_o)$, Lebesgue's Dominated Convergence Theorem implies that $D\, \mathbb{I}_{M_{o}\setminus\mathrm{Crit}(u)}\in L^{1}(E_{s}^{t})$ and 
that 
\begin{equation*}
\lim_{k\to + \infty} \!\!\!\!\!\! \int\limits_{\{1-\frac{1}{s}< u<1-\frac{1}{t}\}}\!\!\!\!\!\!\!\!\!\! D_k \, d\mu\,\,\,\,\,\, =  \!\!\!\!\!\!\!\!\!\!\!\!\int\limits_{\{1-\frac{1}{s}< u<1-\frac{1}{t}\} \setminus \mathrm{Crit}(u)} \!\!\!\!\!\!\!\!\!\!\!\!\!\!\!\! D \, 
d \mu\, < +\infty \,.
\end{equation*}
This fact, combined with inequality~\eqref{tfci}, implies that the sequence of the integrals of the functions $P_k$ is uniformly bounded in $k$. On the other hand, the $P_k$'s are clearly nonnegative and they converge monotonically and pointwise 
to the function $P\,\,\mathbb{I}_{M_{o}\setminus\mathrm{Crit}(u)}$.
The Monotone Convergence Theorem thus yields 
\begin{equation*}
\lim_{k\to + \infty} \!\!\!\!\!\! \int\limits_{\{1-\frac{1}{s}< u<1-\frac{1}{t}\}}\!\!\!\!\!\!\!\!\!\! P_k \, d\mu\,\,\,\,\,\, =  \!\!\!\!\!\!\!\!\!\!\!\!\int\limits_{\{1-\frac{1}{s}< u<1-\frac{1}{t}\} \setminus \mathrm{Crit}(u)}\!\!\!\!\!\!\!\!\!\!\!\!\!\!\! P\, d \mu  \, < +\infty\,.
\end{equation*}
In particular, we have that $P\, \mathbb{I}_{M_{o}\setminus\mathrm{Crit}(u)}\in L^{1}(E_{s}^{t})$.
Since $\mathrm{div}(Y)\mathbb{I}_{M_{o}\setminus\mathrm{Crit}(u)}=(P+D)\mathbb{I}_{M_{o}\setminus\mathrm{Crit}(u)}$, it follows {then that $\mathrm{div}(Y)\mathbb{I}_{M_{o}\setminus\mathrm{Crit}(u)}\in L^{1}_{loc}(M_o)$}, as desired. 

Having the claim~\eqref{eq:claim} at hand, we are now going to prove that $\Phi \in W^{1,1}_{loc}(0,+\infty)$ with weak derivative given by
\begin{align}\label{Phi'}
\Phi'(t)&\,=\!\!\!\!\!\!\int\limits_{\{u=1-\frac{1}{t}\}\setminus\mathrm{Crit}(u)}\!\!\left[-\,\frac{\rm{R}^{\Sigma_{t}}}{2} \, + \, \frac{\vert\,\nabla^{\Sigma_t}\vert \nabla u\vert\,\vert^{2}}{\vert \nabla u \vert^{2}}+\frac{\Ro}{2}+\frac{\,\vert \mathring{\mathrm{h}}\vert^{2}}{2}\, +\frac{3}{4}\left(\,\frac{2\vert \nabla u \vert}{1-u} -\mathrm{H}\right)^{2}\,\right] d\sigma
\end{align}
a.e. in $(0,+\infty)$. The latter is in $L^1_{loc}(0,+\infty)$, thanks precisely to claim~\eqref{eq:claim} coupled with the elementary properties of the integrals and the Coarea Formula. 
The following argument is inspired by~\cite{Ben_Fog_Maz_1}. Let $\chi$ be a test function belonging to $C_{c}^{\infty}(0,+\infty)$. We have that
\begin{align}
\int\limits_{0}^{+\infty}\!\!\chi'(\tau)\,\Phi(\tau)\,d\tau&\,=\,\int\limits_{0}^{+\infty}d\tau\!\!\!\!\!\!\int\limits_{\big\{\frac{1}{1-u}=\tau\big\}\setminus\mathrm{Crit}(u)}\!\!\! \!\!\!\!\!\!\!\!\!\chi'\Big( \frac{1}{1-u} \Big)\,\frac{\langle Y,\nabla u\rangle}{\vert \nabla u \vert}\,d\sigma\,=\!\!\!\int\limits_{M_{o}\setminus\mathrm{Crit}(u)}\!\!\!\Big\langle Y,\nabla \Big[\chi\,\Big( \frac{1}{1-u} \Big)\Big] \Big\rangle\,d\mu\nonumber\\
&\,=\,\lim_{k\to +\infty}\int\limits_{M_{o}}\Big\langle Y_{k},\nabla \Big[\chi\,\Big( \frac{1}{1-u} \Big)\Big] \Big\rangle\,d\mu\,=\,-\,\lim_{k\to +\infty}\int\limits_{M_{o}}\chi\,\Big( \frac{1}{1-u} \Big)\,\mathrm{div}( Y_{k})\,d\mu\nonumber\,,
\end{align}
where the second equality follows by the Coarea Formula, the third one by Lebesgue's Dominated Convergence Theorem, whereas the last one is a simple integration by parts. We now let $0<s<t<+ \infty$ be such that $\mathrm{supp}\chi \subset (s,t)$. Setting 
\begin{equation}
Q\,:=\, \frac{|\na u|}{(1-u)^{2}}\,\left\vert\,\frac{ \nabla u}{1-u}+ \frac{\nabla\vert \nabla u\vert}{|\na u|}\,\right\vert^{2}\, \nonumber
\end{equation}
on $M_{o}\setminus\mathrm{Crit}(u)$ and using formula~\eqref{eq2}, one gets 
\begin{align}
\int\limits_{M_{o}}\!\chi\,\Big( \frac{1}{1-u} \Big)\,\mathrm{div}( Y_{k})\,d\mu&\,=\int\limits_{E_{s}^{t}}\!\chi\,\Big( \frac{1}{1-u} \Big)\,\Big\{P_{k}\,+\,D_{k}\,+\, \frac{|\na u|}{1-u}\,\, \eta'_{{k}}\Big(\frac{\vert \nabla u\vert}{1-u}\Big)\,Q\Big\}\,d\mu\,.\nonumber
\end{align}
Using the elementary inequality~\eqref{eq:hdubound}, it is not hard to show that 
$$
Q\,\mathbb{I}_{M_{o}\setminus\mathrm{Crit}(u)} \, \leq \, {3 P \,\mathbb{I}_{M_{o}}\setminus\mathrm{Crit}(u)} + 2 \frac{|\na u||\na\na u|}{(1-u)^3}\, ,
$$
so that $Q\,\mathbb{I}_{M_{o}\setminus\mathrm{Crit}(u)} \in L^1_{loc}(M_o)${, whereas $|\na u|/(1-u)\,\, \eta'_{{k}}\big(\vert \nabla u\vert/(1-u)\big)$ is always bounded}. 
As $\lim_{k \to + \infty}\eta_k'(\tau)=0$ for every $\tau \in (0, + \infty)$, the Dominated Convergence Theorem implies that 
$$
\lim_{k\to +\infty}\int\limits_{E_{s}^{t}}\!\!\chi\,\Big( \frac{1}{1-u} \Big)\,\frac{|\na u|}{(1-u)}\,\, \eta'_{{k}}\Big(\frac{\vert \nabla u\vert}{1-u}\Big)\,Q\,d\mu\,=\,0.
$$
All in all we get
\begin{align}
&\int\limits_{0}^{+\infty}\!\!\!\chi'(s)\,\Phi(s)\,ds\,=\,-\,\lim_{k\to +\infty}\int\limits_{M_{o}}\chi\,\Big( \frac{1}{1-u} \Big)\,\mathrm{div}( Y_{k})\,d\mu=-\!\!\!\!\!\int\limits_{M_{o}\setminus \mathrm{Crit}(u)}\!\!\!\!\!\!\chi\,\Big( \frac{1}{1-u} \Big)\,\mathrm{div}( Y)\,d\mu\,\nonumber\\
&\,=-\!\!\int\limits_{0}^{+\infty}\!\!\chi(s) \,\,  ds\!\!\! \!\!\!\!\!\!\!\int\limits_{\{u=1-\frac{1}{s}\}\setminus\mathrm{Crit}(u)}\!\!\left[\!-\frac{\,\rm{R}^{\Sigma_{s}}}{2}+\frac{\vert\,\nabla^{\Sigma_s}\vert \nabla u\vert\,\vert^{2}}{\vert \nabla u \vert^{2}}+\frac{\Ro}{2}+\frac{\,\vert \mathring{\mathrm{h}}\vert^{2}}{2}\, +\frac{3}{4}\left(\,\frac{2\vert \nabla u \vert}{1-u} -\mathrm{H}\right)^{2}\right]d\sigma,\nonumber
\end{align}
where in the last identity we used~\eqref{div(Y)geom} together with the Coarea Formula. It is now clear that $\Phi\in W^{1,1}_{loc}(0,+\infty)$. As $F(t)=4\pi t+\Phi(t)$, it is also clear that $F$ belongs to $W^{1,1}_{loc}(0,+\infty)$ and that its weak derivative coincides a.e. with the expression in~\eqref{F'}. In particular, the function $F$ admits a (locally) absolutely continuous representative -- still denoted by $F$ -- satisfying the following identity, for any pair of positive real numbers $s < t$
\begin{align}
&F(t)-F(s) \,\, = \nonumber\\
&=\!\int\limits_{s}^{t} d\tau \!\left\{  4\pi +\!\!\!\!\!\!\int\limits_{\{u=1-\frac{1}{\tau}\}\setminus\mathrm{Crit}(u)}\!\!\!\left[-\,\frac{\rm{R}^{\Sigma_{\tau}}}{2} + \frac{\vert\,\nabla^{\Sigma_\tau}\vert \nabla u\vert\,\vert^{2}}{\vert \nabla u \vert^{2}}+\frac{\Ro}{2}+\frac{\,\vert \mathring{\mathrm{h}}\vert^{2}}{2}\, +\frac{3}{4}\left(\,\frac{2\vert \nabla u \vert}{1-u} -\mathrm{H}\right)^{2}\,\right] d\sigma\right\}\nonumber\\
  &=\!\!\!\!\!\int\limits_{\left[s , t \right] \setminus \, \mathcal{N}} \!\!\! d\tau \! \left\{  4\pi -\!\!\!\!\!\int\limits_{\{u=1-\frac{1}{\tau}\}}\!\!\!\!\!\frac{\,\,\rm{R}^{\Sigma_{\tau}}}{2} \, d\sigma +\!\!\!\!\!\!\int\limits_{\{u=1-\frac{1}{\tau}\} }\!\!\left[ \frac{\vert\,\nabla^{\Sigma_{\tau}}\vert \nabla u\vert\,\vert^{2}}{\vert \nabla u \vert^{2}}+\frac{\Ro}{2}+\frac{1}{2} \vert \mathring{\mathrm{h}}\vert^{2}+\frac{3}{4}\left( \frac{2\vert \nabla u \vert}{1-u} -\mathrm{H}\right)^{\!\!2}\right]
  d\sigma \! \right\}, 
\end{align}
where $\mathcal{N}$ is the set of the critical values of $u$ and it is negligible, by Sard's Theorem. Since all the level sets corresponding to regular values of $u$ are closed regular surfaces, the monotonicity of $F$ follows by the very same considerations made after formula~\eqref{monoliscia}, with just one important exception. Indeed, in the present context, we need a different argument to ensure that the regular level sets of $u$ are connected, so that, by the Gauss-Bonnet Theorem, one obtains the inequality $4\pi - 2\pi \chi(\{ u=1-\frac{1}{\tau} \}) \geq 0$, for every $\tau \in (0,+\infty)\setminus \mathcal{N}$. Here is where the assumption $H_2(M;\Z)=\{0\}$ comes into play. 
To see this, suppose by contradiction that for some $\tau \in (0,+\infty)\setminus\mathcal{N}$ the (regular) level set $\Sigma = \{ u=1-\frac{1}{\tau}\}$ is given by the disjoint union of at least two connected components $\Sigma'$ and $\Sigma''$, so that $\Sigma'$ and $\Sigma''$ are two closed hypersurfaces. By the triviality of $H_2(M;\Z)$, we have that each closed 2-dimensional surface in $(M,g)$ is the boundary of a 3-dimensional bounded open domain. In particular, there exist two bounded connected and open subsets $\Om', \Om'' \subset M$ such that $\partial \Om' = \Sigma'$ and $\partial \Om'' = \Sigma''$. If $o$ doesn't belong to $\Om'$, then $\overline{\Om'}$ is contained in $M_{o}$, and  by the Strong Maximum Principle $u$ must then be constant in $\overline{\Om'}$. But this is impossible, as all the level sets of $u$ have finite $2$-dimensional Hausdorff measure, since $u$ is harmonic. The same argument obviously applies to $\Om''$; therefore, the pole $o$ must belong to both $\Om'$ and $\Om''$, so that these latter turn out to have a nonempty intersection. Consequently, either $\Om'\subset \Om''$ or $\Om''\subset \Om'$, since $\partial \Om' = \Sigma'$ and $\partial \Om'' = \Sigma''$ are disjoint. On the other hand, by the Strong Maximum Principle $u$ must then be constant either in $\overline{\Om''}\setminus \Om'$ or in $\overline{\Om'}\setminus \Om''$ respectively, which is again a contradiction.
\end{proof}
Combining the above theorem with some standard facts about the asymptotic behaviour of the Green's functions near the pole, one gets the following corollary, that should be regarded as an abstract Positive Mass Theorem, holding in a more general framework, where the total mass is replaced by the quantity $\lim_{t \to + \infty} F(t)$, or more likely, {by $\sup_{o \in M} \lim_{t \to + \infty} F_o(t) $, where the index $o$} here is reminiscent of the pole of the Green's function that we are considering. 
\begin{corollary}
\label{positiveAOM}
Under the assumptions of Theorem~\ref{effectivemonotonicity}, we have that 
\begin{equation}
\label{eq:PMTAOM}
0 \, \leq \, \lim_{t \to + \infty} F(t) \, .
\end{equation}
Moreover, if $\lim_{t \to + \infty} F(t) = 0$, then $(M,g)$ is isometric to $(\R^3, g_{\R^3})$.
\end{corollary}
\begin{proof}
We first claim that $\lim_{t\to 0^+}F(t) = 0$. To see this fact, we recall that $u$ is related to the minimal positive  Green's function $\mathcal{G}_o$ of $(M,g)$ with pole at $o$ through the formula $u=1-4\pi\mathcal{G}_{o}$. 
Consequently, there holds
\begin{equation}\label{eq17}
\int\limits_{\{u=1-\frac{1}{t}\}}\!\!\!\vert \nabla u\vert \, d\sigma \, \equiv  \, 4\pi\,
\end{equation}
for every $t\in (0,+\infty)\setminus\mathcal{N}$.
On the other hand, it is well known ({see for example~\cite{MRS}}) that $\mathcal{G}_o$ displays the following asymptotic behaviour near the pole:
\begin{align*}
\Big| \mathcal{G}_{o} - \frac{1}{4\pi r}\Big|&\, = \, o (r^{-1})\,,
\\
\Big\vert \nabla \mathcal{G}_{o}+\frac{1}{4\pi r^2}\,\nabla r\Big\vert& \,= \, o(r^{-2})\,,
\\
\!\!\!\!\!\!\!\!\!\!\!\!\!\!\!\!\!\!\!\!\Big\vert \nabla \nabla \mathcal{G}_{o}-\frac{1}{4\pi r^2}\,\Big(\,\frac{2}{r} \, dr\otimes dr- \nabla \nabla r\Big)\Big\vert& \, = \, o(r^{-3})\,,
\end{align*}
where $r$ denotes the distance to the pole $o$. As a consequence, in a sufficiently small neighborhood of $o$, 
the function $u$ is subject to the bounds 
$$
\frac{C_1}r\leq 1-u\leq\frac{C_2}r,
\qquad
\frac{C_3}{r^2}\leq\vert \nabla u \vert \leq \frac{C_4}{r^2},
\qquad
\vert \nabla  \nabla u \vert\leq \frac{C_5}{r^3},
$$
for some positive constants $C_{i}>0$, $i=1, \ldots, 5$. Combining these bounds with~\eqref{eq:hdubound} and~\eqref{eq17}, we conclude that 
\begin{align}
&t^{2} \!\!\!\!\!\!\int\limits_{\{u=1-\frac{1}{t}\}}\!\!\!\!\! \vert \nabla u \vert^{2} \, d\sigma \,\, \leq \, \!\!\!\!\!\!\int\limits_{\{u=1-\frac{1}{t}\}} \!\!\!\!\!\frac{C_4}{r^2(1-u)^2}\vert \nabla u \vert \,  d\sigma \, \leq \, \frac{C_4}{C_1^2} \!\!\!\! \int\limits_{\{u=1-\frac{1}{t}\}} \!\!\!\!\!\vert \nabla u \vert \, d\sigma \,\, \leq \, \frac{4 \pi C_4}{C_1^2}\,,\nonumber\\
&t\!\!\!\!\!\! \int\limits_{\{u=1-\frac{1}{t}\}}\!\!\!\!\! \vert \,\mathrm{H}\,\vert\,\vert \nabla u \vert\, \, d\sigma\, \leq \, \!\!\!\!\!\!\int\limits_{\{u=1-\frac{1}{t}\}} \!\!\!\! \frac{\vert \nabla \nabla u\vert}{1-u}\, d\sigma 
\, \leq \, \frac{C_5}{C_1 C_3} \!\!\!\! \int\limits_{\{u=1-\frac{1}{t}\}} \!\!\!\!\!\vert \nabla u \vert \, d\sigma \,\, \leq \, \frac{4 \pi C_5}{C_1 C_3} \, \nonumber.
\end{align}
Plugging these estimates into the definition of $F$, it easily follows that $F(t)\to 0$, as $t \to 0^+$.
Combining this fact with the nonotonicity of $F$ yield \eqref{eq:PMTAOM}.

Let us now focus our attention on the rigidity statement.
By the above discussion, we have that $u$ behaves like $1- 1/r$ and $\na u$ behaves like $\na r/r^2$ in a sufficiently small neighborhood of the pole $o$. In particular, there exists a maximal time $T$ such that $\na u \neq 0$ in $u^{-1}(0,T)$. It follows that $F$ is continuously differentiable on the interval $(0,T)$. Moreover, using~\eqref{eq:PMTAOM}, one easily gets that $F'\equiv 0$ in $(0,T)$, so that all the positive summands in formula~\eqref{monoliscia} are forced to vanish for every $t \in (0,T)$. This fact has very strong implications. First of all $\na^{\Sigma_t}|\na u| \equiv 0$ implies that $|\na u| = f (u)$, for some positive function $f:(0,T) \to (0 , + \infty)$. It turns out that such a function can be made explicit. Indeed, from~\eqref{monoliscia} one also has that $\HHH = 2 f(u)/(1-u)$. On the other hand, it follows from~\eqref{H} that $\HHH = -\pa |\na u|/ \pa u = -f'(u)$. All in all, we have that $f$ obeys the ODE 
$$
f'(u)=- \frac{2f(u)}{1-u} \, .
$$
Now, the only solution to this ODE which is compatible with the asymptotic behaviour of $u$ and $|\na u|$, as $u \to - \infty$, is given by 
$f(u) = (1-u)^2$. Since $u<1$ on the whole manifold, $f$ never vanishes, so that $T= +\infty$ and $|\na u| \neq 0$ everywhere.
In particular, all the level sets of $u$ are regular and diffeomorphic to each other. More precisely, by the vanishing of the Gauss-Bonnet term in~\eqref{monoliscia}, they are all diffeomorphic to a $2$-sphere and $M$ is diffeomorphic to $\R^3$. So far we have that the metric $g$ can be written on $M \setminus \{o\}$ as 
$$
g \, = \, \frac{du  \otimes du}{(1-u)^4} \, +
 \, g_{ij}
(u, \!\vartheta) \, d\vartheta^i  \otimes d\vartheta^j \, ,
$$ 
where $g_{ij}
(u, \!\vartheta) \, d\vartheta^i  \otimes d\vartheta^j$ represents the metric induced by $g$ on the level sets of $u$. Exploiting the vanishing of the traceless second fundamental form of the level sets in~\eqref{monoliscia}, it turns out that the coefficients $g_{ij}(u,\!\vartheta)$ obey the following first order system of PDE's
$$
\frac{\pa g_{ij}}{\pa u}  \, = \, \frac{2 \, g_{ij}}{1-u} \,  \, .
$$
Arguing as in~\cite[Theorem 1.1, Case 2]{Cat_Man_Maz}, 
one can deduce that $g_{ij}
(u, \!\vartheta) \, d\vartheta^i  \otimes d\vartheta^j = (1-u)^{-2} g_{\mathbb{S}^2},$ so that $g$ takes the form
$$
g \, = \, \frac{du  \otimes du}{(1-u)^4} \, + \frac{g_{\mathbb{S}^2}}{(1-u)^2} \, . 
$$
The prescription $\Ro \equiv 0$, which also follows from~\eqref{monoliscia}, implies that $u=1-1/r$ and $g = g_{\R^3}$, if used in combination with the asymptotic behaviour of $u$ and $|\na u|$, as $u \to - \infty$.
\end{proof}

\begin{remark}
\label{ptqlm}
The validity of both Theorem~\ref{effectivemonotonicity} and Corollary~\ref{positiveAOM} might suggest the interpretation of 
the monotonic quantity $F$ as a potential-theoretic notion of quasi-local mass. This impression is confirmed in Section~\ref{sec:ssl}, where the small sphere limit of the monotonic quantities is analysed in details. 
\end{remark}


\section{Proof of the positive mass theorem}


Building on Theorem~\ref{effectivemonotonicity}, we present in this section a new proof of the (Riemannian) Positive Mass Theorem, originally due to Schoen and Yau~\cite{Sch_Yau_1979,Sch_Yau_1981} and Witten~\cite{Witten}. 
The precise statement goes as follows:
\begin{theorem}[Positive Mass Theorem]
\label{PMT}
Let $(M,g)$ be a complete, asymptotically flat, $3$-dimensional Riemannian manifold with nonnegative scalar curvature. Then, the total ADM-mass of $(M,g)$ is nonnegative. In symbols,
\begin{equation*}
m_{ADM}(M,g) \geq 0.
\end{equation*}
Moreover, $m_{ADM}(M,g)=0$ if and only if $(M,g)$ is isometric to 
$(\R^{3},g_{\R^3})$.
\end{theorem}

In what follows we just discuss the case where $(M,g)$ has only one end, 
from which the general statement can be deduced by nowadays standard arguments. For the reader's convenience, we recall that a complete, $3$-dimensional, one-ended Riemannian manifold $(M,g)$ is said to be {\em  asymptotically flat (of order $\tau$)} if there exists a compact set $K$ such that $M\setminus K$ is diffeomorphic to the exterior of a closed ball $\R^{3}\setminus \overline{B}$, through a so called {\em asymptotically flat coordinate chart} $x=(x^{1},x^{2},x^{3})$. In such a chart, the metric coefficients obey the expansion
\begin{equation}
g_{ij}(x)=\delta_{ij} + O_{2}(\vert x \vert ^{-\tau}),
\end{equation}
for some constant $\tau> 1/2$. Moreover, we assume that the scalar curvature is integrable on $M$. These latter conditions guarantee that the total ADM-mass of $(M,g)$, which can be computed through the chart $x$ as 
\begin{equation}
m_{ADM} (M,g) \, = \, \lim_{r\to +\infty}\, \frac{1}{16\pi} \int\limits_{\{\vert x\vert=r\}} 
 \sum_{i,j=1}^n \left( 
\frac{\partial g_{ij}}{\partial x^j}  
- \frac{\partial g_{jj}}{\partial x^i} \!
\right) \!
\frac{x^i}{|x|} \,\,  d\sigma_{\text{eucl}}
\,\label{eq4},
\end{equation}
is a well defined geometric invariant, whose value does not depend on the particular {\em asymptotically flat coordinate chart} employed in the above computation (see e.g.~\cite{Bartnik}).

\smallskip
Before proceeding with our proof of Theorem~\ref{PMT}, let us mention that since the original work of Schoen and Yau, several other approaches have been proposed to prove this crucial result.
Far from being complete and referring the reader to~\cite{Lee_book} for a comprehensive survey on this topic, we just mention that the first alternative proof was found by Witten~\cite{Witten}, using harmonic spinors. Another route to the Positive Mass Theorem was subsequently provided by the Huisken-Ilmanen's theory of the Weak Inverse Mean Curvature Flow~\cite{Hui_Ilm_2001}, in combination with the Monotonicity of the Hawking Mass observed by Geroch~\cite{Ger}. Yet another proof of the Positive Mass Theorem has been recently proposed by Li~\cite{Li}, using the Ricci Flow. Finally, in a {recent paper}, Bray, Kazaras, Kuhri and Stern~\cite{Bra_Kaz_Khu_Ste_2019} were able to provide a new argument, based on the study of the level sets of linearly growing harmonic functions combined with integral identities, deduced via the Bochner technique. This latter approach, inspired by~\cite{Stern}, turned out to be flexible enough to also allow for the treatment of the space-time case~\cite{Hir_Kaz_Khu,Bra_Hir_Kaz_Khu_Zha_2021}, and, together with some of the computations carried out in~\cite{Jez_Kij,Jez_1989}, is probably the method that displays the largest number of analogy with ours. On this regard, it should be mentioned that level set methods, combined with Bochner technique and integral identities, have recently found some applications to the study of static metrics in general relativity~\cite{Ago_Maz_CMP,Ago_Maz_Oro_1}, even when the function under consideration is not necessarily harmonic~\cite{Bor_Maz_1,Bor_Maz_2,Bor_Chr_Maz}.

\begin{proof}
For the sake of notation, let us set $m = m_{ADM} (M,g)$, and let us focus on the first part of the positive mass statement, i.e., $m\geq 0$.
By the very accurate discussion given in~\cite[Section 2]{Bra_Kaz_Khu_Ste_2019}, and in particular from~\cite[Proposition 2.1]{Bra_Kaz_Khu_Ste_2019}, the analysis can be reduced to the case where the underlying manifold $M$ is diffeomorphic to $\R^3$ and there exists a distinguished asymptotically flat coordinate chart $x=(x^1,x^2,x^3)$ -- called {\em Schwarzschildian coordinate chart} -- in which the metric $g$ can be expressed as 
\begin{equation}
\label{eq:schwas}
g \, = \, \Big(1+\frac{m}{2\vert x\vert}\Big)^{\!4}\,\delta_{ij} \,\, dx^{i}\!\otimes dx^{j}.
\end{equation}
A manifold of this kind clearly fulfills all the assumptions made in Theorem~\ref{effectivemonotonicity}. In particular, there exists {the solution $u$ to~\eqref{f0}, and consequently the function $F$ defined as in~\eqref{eq0}.} 
Also, in virtue of the monotonicity of $F$ established 
in Theorem~\ref{effectivemonotonicity},  we have that 
\begin{equation}
\label{eq18}
\lim\limits_{t\to 0^+}F(t) \,\, \leq \, \lim\limits_{t\to+\infty} \!\!F(t)\,.
\end{equation}
As we have seen in the proof of Corollary~\ref{positiveAOM}, it is a general fact that $\lim_{t\to 0^+}F(t) = 0$. We now claim that $\lim_{t\to+\infty} \!F(t)= 8 \pi m$. It is clear that combining this claim with~\eqref{eq18}, one easily gets $m \geq 0$. In order to compute the limit of $F(t)$ as $t \to + \infty$, it is useful to understand the asymptotic behaviour of $u$ at infinity. This is done in the following lemma.

\begin{lemma}\label{espAScC}
The function $u$ satisfies the asymptotic expansion at infinity
\begin{align}
u& \, = \, 1-\frac1{\vert x\vert}+\frac1{2\vert x\vert^{2}}\,
\big(m+\phi (x/\vert x\vert)\big)+O_2\big(\vert x\vert^{-3+	\alpha} \big) \, ,\label{expansionofu}
\end{align}
where $x= (x^1,x^2,x^3)$ is a Schwarzschildian coordinate chart at infinity, $\phi$ satisfies $\Delta_{\SSS^{2}}\phi=-2\phi$, and $0<\alpha<1/2$ is a fixed real number that can be chosen as small as desired.
\end{lemma}
\begin{proof}[Proof of Lemma~\ref{espAScC}]
First of all we notice that by virtue of the Schwarzschildian asymptotics~\eqref{eq:schwas}, our metric is asymptotically flat of order $\tau=1$. In particular, by~\cite[Appendix A]{MMT}, we know that 
\begin{equation}\label{eq7}
u(x)=1-\frac{\mathcal{C}}{\vert x\vert}+h(x),
\qquad \text{with}\quad h(x)=O_{2}\big(\vert x\vert^{-2 + \alpha}\big)\,,
\end{equation}
for some fixed $0<\alpha<1/2$, that can be chosen as small as desired. Moreover, we have that $\mathcal{C}=1$, in view of~\eqref{eq17}.
Our aim is to investigate the structure of the remainder $h$, eventually proving that
$$
h(x) \, = \, \frac{1}{2\vert x\vert^{2}}\,\big(m+\phi (x/\vert x\vert)\big)+O_{2}\big(\vert x\vert^{-3 + \alpha} \big) \, . 
$$
Writing the equation $\Delta u=0$ in terms of the background Euclidean metric $g_{\R^3}$, it is readily seen that $h$ satisfies the equation
\begin{equation}
\label{eqne_h}
\Delta_{\R^3} h(x)\, - \,\mathcal{U}^{-1}\frac{m}{\vert x\vert^{2}} 
\, \left\langle \! \mathrm{D}h(x) , \frac{x}{\vert x\vert}\right\rangle_{\!\!\R^3} 
\!\!=\, 
\mathcal{U}^{-1}\,\frac{m}{\vert x\vert^{4}},
\end{equation}
where we have set 
$\mathcal{U}=1+\frac{m}{2\vert x\vert}$, for the ease of the reader. 
To proceed, we plug into \eqref{eqne_h} the {\em ansatz} 
\begin{align}\label{eq8}
h(x)=\mathcal{U}^{-1}\frac{m}{2\vert x\vert^{2}}+f(x)\,,
\end{align}
with $f(x)=O_{2}\big(\vert x\vert^{-2 + \alpha}\big)$
and find out that the new remainder $f$ satisfies the  equation
\begin{align}
\label{eq:remf}
\Delta_{\R^3} f(x)\, = \,\mathcal{U}^{-1}\frac{m}{\vert x\vert^{2}} 
\, \left\langle \! \mathrm{D}f(x) , \frac{x}{\vert x\vert}\right\rangle_{\!\!\R^3}
.
\end{align}
Observing that $\mathcal{U} \to 1$ as $|x|\to +\infty$, our claim on $f$ is that 
$$
f(x) \, = \, \frac{\phi (x/\vert x\vert)}{2\vert x\vert^{2}}\,+O_{2}\big(\vert x\vert^{-3+ \alpha} \big) \, . 
$$
To prove this claim, we observe that, since $f(x)=O_{2}\big(\vert x\vert^{-2+ \alpha}\big)$, the right hand side of~\eqref{eq:remf} can be estimated 
by an $O_{1}\big(\vert x\vert^{-5 + \alpha}\big)$. Moreover, among the elements lying in the kernel of $\Delta_{\R^3}$ which are compatible with the condition $f(x)=O_{2}\big(\vert x\vert^{-2 + \alpha}\big)$, the ones with the slowest decay rate are of the form ${\phi (x/\vert x\vert)}/({2\vert x\vert^{2}})$, for some $\phi$ as in the statement of the lemma, whereas the other ones 
can be estimated by an $O_{2}\big(\vert x\vert^{-3}\big)$. {The claim, as well as the thesis of the lemma, follows now from the standard theory for elliptic PDEs.} 
\end{proof}

With Lemma~\ref{espAScC} at hand, we can now compute the limit on the right hand side of~\eqref{eq18}. To this aim, it is convenient to rewrite $F$ as 
$$
F(t) \, = \!\!\!\!\!\!\int\limits_{\{u=1-\frac{1}{t}\}} \!\!\!\! \frac{1}{1-u}  \left[ 1 {+
\frac{\langle \nabla \vert \nabla u \vert,\!\nabla u \rangle}{(1-u)\, \vert \nabla u \vert^{2}}}
+\frac{\vert \nabla u \vert}{(1-u)^{2}}\right] \vert \nabla u \vert \,d\sigma \,.
$$
By virtue of Lemma~\ref{espAScC}, we have that 
\begin{align}
\vert \nabla u\vert&\, = \, \frac{1}{\,\,\vert x\vert^{2}}\left[\,1-\frac{1}{\vert x\vert}\big(2m+\phi (x/\vert x\vert)\big)+O\big(\vert x\vert^{-2+\alpha}\big)\,\right]\,,\nonumber\\
\frac{ \langle \nabla \vert \nabla u \vert, \!\nabla u \rangle}{\vert \nabla u \vert^{2}}&\, = \, -\frac{2}{\vert x\vert}\left[\,1-\frac{1}{2\vert x\vert}\big(4m+\phi (x/\vert x\vert)\big)+O\big(\vert x\vert^{-2+\alpha}\big)\,\right]\nonumber \,,
\end{align}
so that 
$$
\lim_{|x|\to + \infty}\frac{1}{1-u}  \left[ 1 -
\frac{\langle \nabla \vert \nabla u \vert,\!\nabla u \rangle}{(1-u)\, \vert \nabla u \vert^{2}}
+\frac{\vert \nabla u \vert}{(1-u)^{2}}\right] \, = \, 2m \, .
$$
In particular, for every $\ep>0$, there exists $t_\ep>0$ such that whenever {$u\geq 1-1/t_\ep$} one has that
$$
2m-\ep \, \leq \, \frac{1}{1-u}  \left[ 1 -
\frac{\langle \nabla \vert \nabla u \vert,\!\nabla u \rangle}{(1-u)\, \vert \nabla u \vert^{2}}
+\frac{\vert \nabla u \vert}{(1-u)^{2}}\right] \! (x)\, \leq \, 2m + \ep \, .
$$
Using this fact in combination with~\eqref{eq17}, we deduce that, for every $t\geq t_\ep$, it holds
$$
4\pi(2m-\varepsilon)
\,\leq\, 
F(t)
\,\leq\, 
4\pi(2m+\varepsilon).
$$
Therefore, we have that $\lim_{t \to + \infty}F(t) = 8\pi m$ and 
in turn that $m\geq0$.

The rigidity statement, i.e. the fact that $(M,g)$ and $(\R^{3},g_{\R^3})$ are isometric if $m=0$, can be deduced from the validity of the inequality $m \geq 0$, through the nowadays standard argument proposed in the original Schoen-Yau's paper~\cite{Sch_Yau_1979} (see also~\cite[pp.95-97 and p.102]{Lee_book}).
\end{proof}


\section{Small sphere limits and nonnegative scalar curvature}
\label{sec:ssl}


As mentioned in Remark~\ref{ptqlm}, Theorem~\ref{effectivemonotonicity} and Corollary~\ref{positiveAOM} support the notion of the quantity $F$, introduced in~\eqref{f0}, as a quasi-local mass. This section provides additional evidence in favour of this idea. Specifically, we will calculate the small sphere limits of our quantities and observe that they can be utilized to accurately determine, with only a constant factor discrepancy, the scalar curvature's precise value. 
{It is important to note that in General Relativity this geometric invariant coincides for time-symmetric initial data sets with mass/energy density.} Therefore, it is reassuring to contemplate that the very same quantities that allow us to establish the non-negativity of the total mass are inherently connected to mass density. Before delving into the explicit computations, let us briefly review the well-established case of the Hawking mass, which serves as a prominent example of quasi-local mass. The Hawking mass is defined as follows:
%
%
%
%
%
$$
m_H (\Sigma) = \sqrt{\frac{|\Sigma|}{16 \pi}} \left( 1 - \frac{1}{16 \pi} \int_\Sigma \HHH^2 \,  d \sigma \right) \, ,
$$
where $\Sigma$ is a smooth closed surface embedded in a $3$-dimensional Riemannian manifold $(M,g)$. In~\cite{FST} the following small-sphere limit of the Hawking mass was computed, leading to
\begin{equation*}
m_H \big(\pa B_q(t)\big) \, = \, \frac{\RRR(q)}{12} \, t^3 + \left( \frac{\Delta \RRR (q)}{120} - \frac{\RRR^2(q)}{144}\right) \, t^5 + O(t^6) \, , \qquad \hbox{as $t \to 0^+$,}
\end{equation*}
where $q$ is a point in $M$, $B_q(t)$ is the ball of radius $t$ centred at $q$, and $\RRR$ denotes the scalar curvature of $(M,g)$. We are going to prove that a similar expansion holds for (a localised version of) the monotonic quantities introduced in Section~\ref{effmon}. To this aim, let us consider a local Green's function $\mathcal{G}_{q}$ with pole at $q \in M$, i.e., the unique distributional solution (see~\cite{aubin1}) to the problem 
\begin{equation}\label{f2}
\begin{cases}
\Delta \mathcal{G}_{q}\,=\,-\delta_{q}\ &\mathrm{in} \ B_q(r_q) \, ,\\
\quad \,\,\mathcal{G}_{q}\,=\,0 &\mathrm{on} \  \partial B_q(r_q) \, ,
\end{cases}
\end{equation}
where we agree that $r_q= \mathrm{inj}(q)/2$ if $\mathrm{inj}(q)<+ \infty$, and $r_q=1$ otherwise.
Here, of course, $\mathrm{inj}(q)$ denotes the injectivity radius at $q$. In particular, the distance function $d_{g}(q,\cdot)$ is smooth on the punctured ball $B_q^*(r_q) = B_q(r_q) \setminus \{ q\}$.
Following~\cite[Proposition B.1]{LiZhu}, it is immediate to check that if $(x^1,x^2,x^3)$ is a normal coordinate system centred at $q$, then $\mathcal{G}_q$ satisfies the asymptotic expansion
\begin{equation}
\label{asymptoticbehaviourofGq}
\mathcal{G}_{q}\,=\,\frac{1}{4\pi \vert x\vert}\,+\,A_{q}\,+\,O_{2}(\vert x\vert^{1-\tau}) \, , \qquad \hbox{as $|x| \to 0$\,,}
\end{equation}
where the exponent $\tau$ can be chosen arbitrarily in $(0,1)$. Notice that, since $\vert x(\cdot)\vert=d_{g}(q,\cdot)$, the constant $A_q$ remains the same regardless of the set of normal coordinates centered at $q$ that is being used. To introduce the localised analog of the monotonic quantity $F$ defined in~\eqref{f0}, it is convenient to consider, on the punctured (closed) ball 
$\overline{B}_{r}^{\,*}(r_q)$, the function
\begin{equation}
\label{definitionuq}
u_{q}\,:=\,1 + 4 \pi A_q - 4\pi \mathcal{G}_{q}\,.
\end{equation}
By the maximum principle, the range of $u_q$ coincides with the interval $(-\infty, 1+ 4\pi A_q]$. In analogy with~\eqref{f0}, we set
\begin{equation}\label{ffeq0}
F_{q}(t)\,\, = \,\, 4\pi t \,\,-\,\,\, t^{2} \!\!\!\int\limits_{\{u_{q}=1-\frac{1}{t}\}}\!\!\!\vert \nabla u_{q} \vert\, \mathrm{H} \,\, d\sigma \, \, + \,\, t^{3} \!\!\!\int\limits_{\{u_{q}=1-\frac{1}{t}\}}\!\!\!\vert \nabla u_{q} \vert^{2} \,\, d\sigma \,,
\end{equation}
for $t \in I_q$, where the set $I_q$ depends on the sign of the constant $A_q$. More precisely, if $A_q <0$, we let $I_q = (0, -1/4\pi A_q)$, whereas if $A_q \geq 0$, we let $I_q = (0, +\infty)$. 
Notice that this choice might appear slightly restrictive in the case where $A_q>0$, as it corresponds to consider only the level sets of $u_q$ with values in $(-\infty, 1)$. Indeed, it would be possible to extend the definition of $F_q$ to the set $(-\infty, - 1/ 4 \pi A_q)$ in order to cover the whole natural range of $u_q$. However, since we are only interested in the small sphere limit of our quantities, this will result in an unnecessary complication. 
With these notations at hand we are now ready to state the main result of this section.
\begin{theorem}[Small Sphere Limit]
\label{thm:ssl}
Let $(M,g)$ be a {complete, $3$-dimensional Riemannian manifold.} Then, at every point $q \in M$, the following expansion holds
\begin{equation}\label{ffeq20}
\frac{F_{q}(t)}{8\pi}\,=\,\frac{\Ro(q)}{12}\,t^3+o(t^{3}) \, , \qquad \hbox{as $t \to 0^+$} \, ,
\end{equation}
where $\RRR$ is the scalar curvature of $g$ and $F_q$ is the quantity defined in~\eqref{ffeq0}.
\end{theorem}
An immediate consequence of the aforementioned statement is the following characterization of the non-negativity of the scalar curvature in terms of the monotonicity of the functions $F_q$, for $q \in M$.
\begin{corollary}[Scalar Curvature Lower Bound]
{Let $(M,g)$ be a complete, $3$--dimensional Riemannian manifold. }
Then, the scalar curvature $\RRR$ of the metric $g$ is nonnegative on $M$ if and only if for every  $q \in M$ the function $F_q$ is either nonnegative or monotonically non decreasing on an initial portion $(0, \ep_q]$ of $I_q$, for some $\ep_q >0$.
\end{corollary}
It is worth noticing that the above characterisation might suggest a plausible notion of synthetic scalar curvature lower bound, to be empoyed in the context of nonsmooth differential and Riemannian geometry. 

\begin{proof}[Proof of Theorem~\ref{thm:ssl}]
As the point $q$ will remain unchanged throughout the proof, we drop it from the notation whenever it is possible. Our first claim is that
\begin{equation}\label{ffeq1}
\lim_{t\to 0^+}\frac{F(t)}{\vert \Omega(t)\vert}\,=\,\frac{\mathrm{R}(q)}{2}\,,
\end{equation}
where $\Omega(t)= \{u<1-1/t\}$ and $\vert  \Omega(t)\vert$ denote its  measure.  In a given normal coordinates system $(x^1,x^2,x^3)$ centred at $q$, so that~\eqref{asymptoticbehaviourofGq} hold, we have that the following expansions are in force
\begin{align}
1-u\,&=\,\frac{1}{\vert x\vert}\,\big(1+O(\vert x\vert^{2-\tau})\big)\,,\label{ffeq9}\\
\partial_{\alpha}u\,&=\,\frac{1}{\vert x\vert^{2}}\bigg[\frac{x_\alpha}{\vert x\vert}+O(\vert x\vert^{2-\tau})\bigg]\,,\label{ffeq14}\\
\vert \nabla u \vert\,&=\,\frac{1}{\vert x\vert^2}\,\big(1+O(\vert x\vert^{2-\tau})\big)\,,\label{ffeq10}\\
\nabla_\alpha\nabla_{\beta}u&\,=\,\frac{1}{\vert x\vert^{3}}\bigg[\delta_{\alpha \beta}-\frac{3x_\alpha x_\beta}{\vert x\vert^2}+O(\vert x\vert^{2-\tau})\bigg]\,,\label{ffeq15}\\
\mathrm{H}\,&=\,\frac{2}{\vert x\vert}\,\big(1+O(\vert x\vert^{2-\tau})\big)\,,\label{ffeq13}
\end{align}
where the function $u$ is defined as in~\eqref{definitionuq}, and we have set {$x_\alpha = \delta_{\alpha \beta} x^\beta$}. Moreover, as far as distance, area and volume are concerned, it is easily seen that, for every $p \in \Sigma_t= \{ u = 1-1/t\}$,
\begin{equation}
\label{ffeq3}
A_{1} t \,\leq\,\vert x(p)\vert\,\leq \,A_{2} t \quad \hbox{and} \quad B_1 t^2 \leq |\Sigma_t| \leq B_2 t^2 \quad \hbox{and} \quad C_1 t^3 \leq |\Omega (t)| \leq C_2 t^3\, ,
\end{equation}
as expected. Arguing as in the proof of Corollary~\ref{positiveAOM}, it is not hard to show  that $\lim_{t \to 0^+} F(t) = 0$. By l'H\^opital's rule, we have that Claim~\eqref{ffeq1} is equivalent to 
\begin{equation}
\label{ffeq7}
\lim_{t\to 0^+}\frac{F'(t)}{\vert \Omega(t)\vert'}\,=\,\frac{\mathrm{R}(q)}{2}\,.
\end{equation}
Notice that, by virtue of~\eqref{ffeq10}, we can always select a small initial interval of times such that the level sets $\Sigma_t$ of $u$ are regular. In the same interval, both the function $t \mapsto F(t)$ and $t \mapsto |\Omega(t)|$ turn out to be continuously differentiable. In particular, we can compute
\begin{align}
F'(t) \, &=  \,  \int\limits_{\Sigma_t}\left[ \, \frac{\vert\,\nabla^{\Sigma_t}\vert \nabla u\vert\,\vert^{2}}{\vert \nabla u \vert^{2}}+\frac{\Ro}{2}+\frac{\,\vert \mathring{\mathrm{h}}\vert^{2}}{2}\, +\frac{3}{4}\left(\,\frac{2\vert \nabla u \vert}{1-u} -\mathrm{H}\right)^{2}\,\right] d\sigma\,,
\nonumber
\\
\vert \Omega(t)\vert'\,&=\,\frac{1}{t^{2}}\int\limits_{\Sigma_t} \frac{d\sigma}{\vert \nabla u\vert} \, = \,  \int\limits_{\Sigma_t} \frac{(1-u)^2}{\vert \nabla u\vert}\,d\sigma\,,
\nonumber
\end{align}
where the second formula is an easy consequence of the Coarea Formula, whereas the first one follows from~\eqref{monoliscia}, taking into account that, for small enough $t$, all the level sets $\Sigma_t$ are spherical, and thus 
$$
 4\pi - \int\limits_{\Sigma_t}\!\frac{\,\rm{R}^{\Sigma_{t}}}{2} \, d\sigma  \, = \, 0 \, ,
$$
by the Gauss-Bonnet Theorem. To compute the limit in~\eqref{ffeq7}, we first observe that
\begin{equation}
\label{ffeq11}
\vert \Omega(t)\vert'\,=\,|\Sigma_{t}| \, (1+O(t^{2-\tau})) \,,
\end{equation}
since, by formulas~\eqref{ffeq9} and~\eqref{ffeq10}, one has
\begin{equation}
\label{espansione utile}
\frac{(1-u)^2}{\vert \nabla u\vert}\,=\,1+O(\vert x\vert^{2-\tau}) \, . \qquad \,\, 
\end{equation}
On the other hand, the simple Taylor expansion of the scalar curvature at $q$ gives
$$
\Ro\,=\,\Ro(q)+O(\vert x\vert)\,,
$$
which implies
\begin{equation} 
\label{ffeq8}
\int_{\Sigma_t}\!\Ro\, d\sigma\,=\,|\Sigma_{t}| \big(\Ro(q)\,+O(t)\big) \, . \quad
\end{equation}
Combinig~\eqref{ffeq11} with~\eqref{ffeq8}, one gets
\begin{equation}\label{ffeq16}
\lim_{t\to 0^+}\frac{ \frac{1}{2}\int_{\Sigma_t}\!\Ro\, d\sigma}{\vert \Omega(t)\vert'}\,=\,\frac{\mathrm{R}(q)}{2}\,.
\end{equation}
Hence, Claim~\eqref{ffeq7} is proven, if we can show that
\begin{equation}
\label{bigo_part}
\int\limits_{\Sigma_t}\left[ \, \frac{\vert\,\nabla^{\Sigma_t}\vert \nabla u\vert\,\vert^{2}}{\vert \nabla u \vert^{2}}+\frac{\,\vert \mathring{\mathrm{h}}\vert^{2}}{2}\, +\frac{3}{4}\left(\frac{2\vert \nabla u \vert}{1-u} -\mathrm{H}\!\right)^{\!\!2}\,\right] \,d\sigma\,=\,O(t^{4-2\tau})\,.
\end{equation}
Actually, we are going to prove that the function in square brackets is of order $O(|x|^{2-2\tau})$, as $|x| \to 0$. The desired estimate will then be a consequence of~\eqref{ffeq3}. To this end, we first observe that the expansions~\eqref{ffeq9},~\eqref{ffeq10} and~\eqref{ffeq13} imply that 
\begin{equation}
\frac{2\vert \nabla u \vert}{1-u} -\mathrm{H}\,=\,O(\vert x\vert^{1-\tau})\,,
\end{equation}
and thus the last summand in the square brackets is of order $O(|x|^{2-2\tau})$. To evaluate the first summand, we recall identity~\eqref{H}
and we observe that
\begin{align}
\nabla^{\Sigma_t}_\alpha |\nabla u| &\, = \,
{ \nabla_\alpha |\nabla u| - \langle   \nabla |\nabla u | , \nu  \rangle \, \nu_\alpha}\, = \, \na_{\!\alpha} |\na u| +  \HHH \, \nabla_{\!\alpha} u \, = \, \frac{\na \na u  \left( {\na u} , \partial_\alpha \right)}{|\na u|} +  \HHH \, \partial_\alpha u \, 
\nonumber\\
& \, = \, -2 \frac{x_\alpha}{|x|^4} + O(|x|^{-1-\tau}) + 2 \frac{x_\alpha}{|x|^4} + O(|x|^{-1-\tau})  \, = \,O(|x|^{-1-\tau}) \, , \nonumber
\end{align}
where in the last row we used the asymptotic expasions~\eqref{ffeq14}, \eqref{ffeq10}, \eqref{ffeq15} and \eqref{ffeq13}. It follows that
\begin{equation*}
\frac{\vert\,\nabla^{\Sigma_t}\vert \nabla u\vert\,\vert^{2}}{\vert \nabla u \vert^{2}} \, = \,  O (|x|^{2-2\tau}) \, .
\end{equation*}
Finally, we need to prove that also $\vert \mathring{\mathrm{h}}\vert^{2}\,=\,\, O(\vert x\vert^{2-2\tau})$. To see this, let $\ep>0$ be a positive real number such that $\ep^2<1/3$ and let $U_\alpha$ be the open set defined by
\begin{equation}
U_{\alpha}\,=\,\bigg\{p\in B_{q}^{\,*}(r_q):\,  \frac{\vert x^{\alpha} \vert}{\vert x\vert} (p)>\ep\,\,\text{and}\,\,\vert \nabla u\vert(p)\neq 0\bigg\}\, \, ,
\end{equation}
for $\alpha = 1,\!...,\!3$. Due to our choice of $\ep>0$, it is easy to realize that 
\begin{equation}\label{ffeq17}
U_{1}\cup U_{2}\cup U_{3}\,=\,\big\{p\in B_{q}^{\,*}(r_q):\, \vert \nabla u\vert(p)\neq 0\big\}\, \, .
\end{equation}
Without loss of generality, let us focus on the open set $U_1$ and prove the desired estimate $\vert \mathring{\mathrm{h}}\vert^{2}\,=\,\, O(\vert x\vert^{2-2\tau})$ on it. To perform our computations, it is convenient to consider on $U_1$, the frame field given by $\{\na u/|\na u|, X_2, X_3\}$, where the vector fields $X_2$ and $X_3$ are defined as
\begin{equation*}
X_{2}\,=\,-\partial_{2}u\,\partial _{1}\,+\,\partial_{1}u\,\partial _{2} \, , \qquad \hbox{and} \qquad   X_{3}\,=\,-\partial_{3}u\,\partial _{1}\,+\,\partial_{1}u\,\partial _{3} \,.
\end{equation*}
It is immediate to check that $\na u / |\na u|$ is orthogonal to both $X_2$ and $X_3$ and that the latter are linearly independent on $U_1$.
To complete the set-up, we let $\{du/|\na u| , \xi^2, \xi^3 \}$ be the dual co-frame of $\{\na u/|\na u|, X_2, X_3\}$. In this framework, {using the expansions~\eqref{ffeq14},~\eqref{ffeq10} and~\eqref{ffeq15}}, we can write
\begin{align}
g_{\Sigma}\,=\,g^{\Sigma}_{ij}\,\xi^{i}\!\otimes \xi^{j}&\,=\,g(X_{i},X_{j})\,\xi^{i} \!\otimes \xi^{j}\,=\,\frac{1}{\vert x\vert^{4}}\bigg[\frac{x_{i}x_{j}}{\vert x\vert^2}+\,\frac{(x^{1})^2}{\vert x\vert^2}\delta_{i j}+O(\vert x\vert^{2-\tau})\bigg]\,\xi^{i}\!\otimes \xi^{j}\,,
\nonumber\\
\mathrm{h}\,=\,\mathrm{h}_{ij}\,\xi^{i}\!\otimes \xi^{j}&\,=\,\frac{\nabla \nabla u(X_{i},X_{j})}{\vert \nabla u\vert}\,\xi^{i}\!\otimes \xi^{j}\,=\,\frac{1}{\vert x\vert^{5}}\bigg[\frac{x_{i}x_{j}}{\vert x\vert^2}+\,\frac{(x^{1})^2}{\vert x\vert^2}\delta_{i j}+O(\vert x\vert^{2-\tau})\bigg]\,\xi^{i}\!\otimes \xi^{j}\,, \nonumber
\end{align}
where $i,j \in \{ 2,3\}$. From the first of above expressions and by the fact that $|x^1| > \ep |x|$, it follows that $g_\Sigma^{ij} = O(|x|^4)$ on $U_1$, and in turn that 
\begin{equation}
\vert \mathring{\mathrm{h}}\vert^{2}\,=\,g_{\Sigma}^{i k}\,g_{\Sigma}^{j l}\Big(\mathrm{h}_{il}-\frac{\mathrm{H}}{2}g^{\Sigma}_{il}\Big)\Big(\mathrm{h}_{jk}-\frac{\mathrm{H}}{2}g^{\Sigma}_{jk}\Big)
\,=\, O(\vert x\vert^{2-2\tau})\,,
\end{equation}
where we also took advantage of expansion~\eqref{ffeq13} in order to get the desired cancellation. Since the same estimate holds on $U_2$ and $U_3$, we finally have that~\eqref{bigo_part} is proven, so that in turn also Claim~\eqref{ffeq1} is proven. Using the third estimate in~\eqref{ffeq3}, we have that Claim~\eqref{ffeq1} can be reformulated as follows
\begin{equation*}
\frac{F_{q}(t)}{8\pi}\,=\,\frac{\Ro(q)}{16\pi}\,\,\vert \Omega(t)\vert+o(t^{3})\,.
\end{equation*}
Therefore, in order to obtain the expansion ~\eqref{ffeq20}, it is sufficient to show that
\begin{equation*}
\label{ffeq22}
\lim_{t\to 0^+}\frac{\vert \Omega(t)\vert}{\frac{4\pi}{3}\,t^3}\,=\,1\,.
\end{equation*}
By l'H\^opital's rule, this reduces to proving that 
\begin{equation*}\label{ffeq22}
\lim_{t\to 0^+}\frac{\vert \Omega(t)\vert'}{{4\pi}\,t^2}\,=\,1\,.
\end{equation*}
The latter statement follows from~\eqref{ffeq11}, combined with 
\begin{equation*}
|\Sigma_t| \, = \int_{\Sigma_t} \!\!d \sigma \, = \int_{\Sigma_t} \!\frac{|\na u|}{(1-u)^2} \, d \sigma \, + \, O(t^{4-\tau}) \, = \, 4 \pi t^2\, + \, O(t^{4-\tau}) \, ,
\end{equation*}
where in the last two identities we used~\eqref{espansione utile} and the fact that $\int_{\Sigma_t} \!|\na u| \,d \sigma = 4 \pi$.
\end{proof}


\section{Further directions: Geroch--type calculation for $p$-harmonic functions.}


In this last section, we extend the monotonicity formula~\eqref{monoliscia} to the nonlinear potential theoretic setting, where harmonic functions are replaced by $p$-harmonic functions, with $1<p<3$. Subsequently, we show how these new monotonicity formulas can be employed to deduce {the Riemannian Penrose Inequality under favourable assumptions.} The treatment of the general case is beyond the purposes of the present note and it is deferred to~\cite{Ago_Maz_oro_3}. We consider, for every $1<p<3$, the unique solution $u$ to the problem
\begin{equation}\label{f1}
\begin{cases}
\Delta_{p} u=0\  &\mathrm{in} \ {M},\\
\quad \,  u=0 &\mathrm{on} \  \partial M,\\
\quad \, u \to 1 &\mathrm{at} \  \infty,
\end{cases}
\end{equation}
where $(M,g)$ is a $3$-dimensional, complete, asymptotically flat Riemannian manifold with nonnegative scalar curvature and {smooth, connected, minimal boundary}. Assume that $|\na u| \neq 0$ everywhere, so that $u$ is smooth.

\begin{remark}
Notice that the condition $|\na u| \neq 0$, though far from being optimal, is automatically satisfied by relevant classes of boundary geometries. For example, following~\cite{Lewis}, one can show that, in the Euclidean setting, mean convex and star-shaped boundaries give rise to $p$-capacitary potentials whose gradient is never vanishing. It is also worth noticing that the monotonicity of the functions $F_{p}$, defined below, still holds if the set of the critical values of the $p$-harmonic function $u$ is negligible and the regular level sets of $u$ are connected.
\end{remark}

To introduce our monotonic quantities, we recall that the $p$-capacity of $\partial M$ is defined as 
\begin{equation}\label{pcapacity}
\mathrm{Cap}_{p}(\partial M):=\inf \Bigg\{ \int\limits_{M}\vert  \nabla v \vert ^{p}\,d\mu:\, v\in C^{\infty}_{c}(M), \, v=1\,\, \text{on}\,\, \partial M\Bigg\}\,,
\end{equation}
which is related to $u$ through the well known identities
\begin{equation}\label{eq24}
\mathrm{Cap}_{p}(\partial M)\, =\int\limits_{M}\vert \nabla u\vert^{p}\,d\mu \, = \!\!\!\!\int\limits_{\{u=t\}} \!\!\!\vert \nabla u\vert^{p-1}\,d\sigma .
\end{equation}
In analogy with~\eqref{eq0}, we define, for every $1<p<3$, the function 
\begin{align}
\label{defFp}
F_{p}(t)&\, = \,\, 4\pi t \, -\, \frac{t^{\frac{2}{p-1}}}{c_{p}} \!\!\!\!\int\limits_{\{u=\alpha_{p}(t)\}} \!\!\!\! \vert \nabla u\vert \, \mathrm{H}\, d\sigma \, + \, \frac{t^{\frac{5-p}{p-1}}}{c_{p}^{2}} \!\!\!\int\limits_{\{u=\alpha_{p}(t)\}} \!\!\!\vert \nabla u\vert^{2} \, d\sigma \, ,
\end{align}
where 
\begin{equation}
\label{eq23}
c_{p}
\,=\,
\Bigg(\frac{\mathrm{Cap}_{p}(\partial M)}{4\pi}\Bigg)^{\!\frac{1}{p-1}},
\qquad\qquad
\alpha_{p}(t)
\,=\,
1-c_{p}\,\,\frac{p-1}{3-p}\,t^{-\frac{3-p}{p-1}},
\end{equation}
and the variable $t$ ranges in $\left[\big(c_{p}\,\,\frac{p-1}{3-p}\,\big)^{\!\frac{p-1}{3-p}}, + \infty \right)$.
Proceeding in the same spirit as in the smooth proof of Theorem~\ref{effectivemonotonicity}, we compute
\begin{align}
\frac{d}{dt} \!\!\!\!\int\limits_{\{u=\alpha_{p}(t)\}} \!\!\!\!\vert \nabla u\vert^{2}\, d\sigma& \, = \, -\,\frac{3-p}{p-1}\,\,\frac{c_{p}}{t^{\frac{2}{p-1}}}\int\limits_{\{u=\alpha_{p}(t)\}}\vert \nabla u\vert\,\mathrm{H} \, d\sigma\,,\nonumber\\
\frac{d}{dt}\!\!\!\!\int\limits_{\{u=\alpha_{p}(t)\}}\!\!\!\!\vert \nabla u\vert \, \mathrm{H}\, d\sigma&\, = \, \frac{c_{p}}{t^{\frac{2}{p-1}}}\!\int\limits_{\{u=\alpha_{p}(t)\}}\Bigg\{\frac{p-2}{p-1}\,\,\mathrm{H}^{2} -
\vert \nabla u \vert  \left[ \, \Delta_{\Sigma_{t}} \!\left(\frac{1}{\vert \nabla u \vert}\right) + \, \frac{\vert \mathrm{h}\vert^{2}+\Ric (\nu,\nu)}{\vert \nabla u\vert} \, \right] 
\Bigg\}d\sigma\,,\nonumber
\end{align}
where $\Sigma_t= \{ u = \alpha_p(t)\}$. Using the traced Gauss equation, as in Section~\ref{effmon}, we obtain 
\begin{align}
\frac{p-2}{p-1}\,\mathrm{H}^{2}&-\vert \nabla u \vert  \left[ \, \Delta_{\Sigma_{t}} \!\left(\frac{1}{\vert \nabla u \vert}\right) + \, \frac{\vert \mathrm{h}\vert^{2}+\Ric (\nu,\nu)}{\vert \nabla u\vert} \, \right] \nonumber\\
&\qquad\qquad = \,\, \Delta_{\Sigma_{t}}(\log \vert \nabla u \vert ) - \frac{\vert\,\nabla^{\Sigma_t}\vert \nabla u\vert\,\vert^{2}}{\vert \nabla u \vert^{2}} - \frac{\Ro}{2} +\!\frac{\,\rm{R}^{\Sigma_{t}}}{2} - \frac{\,\vert \mathring{\mathrm{h}}\vert^{2}}{2} 
- \frac{5-p}{4(p-1)}\,\,\mathrm{H}^{2} .\nonumber
\end{align}
By standard manipulations, we finally get
\begin{equation}
\label{monoliscia_p}
F_p'(t) = 4\pi - \int\limits_{\Sigma_t}\!\!\frac{\,\rm{R}^{\Sigma_{t}}}{2} \, d\sigma  \, + \! \int\limits_{\Sigma_t}\!\left[ \, \frac{\vert\,\nabla^{\Sigma_t}\vert \nabla u\vert\,\vert^{2}}{\vert \nabla u \vert^{2}}+\frac{\Ro}{2}+\frac{\,\vert \mathring{\mathrm{h}}\vert^{2}}{2}\, +\frac{5-p}{4(p-1)}\left[\frac{2(p-1)}{3-p}\frac{\vert \nabla u \vert}{1-u} -\mathrm{H}\right]^{2}\right] d\sigma.
\end{equation}
Again, the right hand side is nonnegative by virtue of Gauss-Bonnet Theorem and by the assumption $\Ro \geq 0$.

{Let us illustrate how the above monotonicity result might be employed to deduce the Riemannian Penrose Inequality.
We start observing} that the monotonicity of $F_p$ implies 
\begin{equation}\label{eq21}
F_{p}(\beta_{p})
\,\leq\,
\lim\limits_{t\to+\infty}F_{p}(t)\,,
\end{equation}
where we have used the short-hand notation 
$\beta_{p} = \big(c_{p}\,\,\frac{p-1}{3-p}\,\big)^{\frac{p-1}{3-p}}$. 
By the minimality of $\pa M$, we immediately get
\begin{equation}\label{eq22}
F_{p}(\beta_{p})\,= \, 4\pi\beta_{p} \,+\,\frac{\beta_{p}^{\frac{5-p}{p-1}}}{c_{p}^{2}}\int\limits_{\partial M}\vert \nabla u\vert^{2} \, d\sigma \, \geq \, 4\pi\beta_{p}\,.
\end{equation}
To compute the limit on the right hand side of~\eqref{eq21}, we 
first observe that $F_p$ rewrites as
\begin{align}
F_{p}(t)
\,\,=\!\!\!\!\!
\int\limits_{\{u=\alpha_{p}(t)\}} \!\!\!\!\!I_{p}\,\vert \nabla u\vert^{p-1}d\sigma\,,\nonumber
\end{align}
where
\begin{align}
I_{p}
&:=c_{p}^{\frac{3p-7}{3-p}}\Big(\,\frac{p-1}{3-p}\,\frac{1}{1-u}\,\Big)^{\!\frac{p-1}{3-p}}\Big[ c_{p}^{3-p}+\Big(\,\frac{p-1}{3-p}\,\frac{c_{p}}{1-u}\,\Big)^{2}\vert \nabla u \vert^{3-p}-\frac{p-1}{3-p}\,\frac{c^{2}_{p}}{1-u}\vert \nabla u\vert^{2-p}\mathrm{H}\Big]\,.\nonumber
\end{align}
To evaluate the asymptotic behaviour of $I_p$, {we assume that, in a distinguished Schwarzschildian coordinate chart $(x^{1},x^{2},x^{3})$ at infinity,} the function $u$ satisfies \begin{equation}\label{asy.exp.ofu}
 u\, = \, 1-\frac{p-1}{3-p}\,\frac{c_{p}}{\vert x\vert^{\frac{3-p}{p-1}}}+\frac{3-p}{2}\,\frac{c_{p}}{\vert x\vert^{\frac{2}{p-1}}}\,\big(m+\phi (x/\vert x\vert)\big)+o_{2}\big(\vert x\vert^{-\frac{2}{p-1}}\big),
\end{equation}
for some $\phi$ satisfying $\Delta_{\SSS^{2}}\phi=-2\phi$. 
The above expansion implies in particular that
\begin{align}
\vert \nabla u\vert&=\frac{c_{p}}{\,\,\vert x\vert^{\frac{2}{p-1}}}\Big[\,1-\frac{2}{p-1}\,\frac{m}{\vert x\vert}-\frac{3-p}{p-1}\,\frac{\phi (x/\vert x\vert)}{\vert x\vert}+o\big(\vert x\vert^{-1}\big)\,\Big]\,,\nonumber\\
\mathrm{H}\, = \,-(p-1)\,\frac{\nabla \nabla u (\nabla u,\nabla u)}{\vert \nabla u\vert^{3}} &=\frac{2}{\vert x\vert}\Big[\,1-\frac{2m}{\vert x\vert}-\frac{3-p}{2}\,\frac{\phi (x/\vert x\vert)}{\vert x\vert}+o\big(\vert x\vert^{-1}\big)\,\Big]\,.\nonumber
\end{align}
In turn, we get
\begin{align}
\lim_{|x|\to + \infty}I_{p}(x) \, = \, 2m\,c_{p}^{1-p} \, .
\end{align}
In other words, for every $\ep>0$, there exists $t_\ep>0$ such that, whenever $u (x)\geq \alpha_p(t_\ep)$, it holds
$$
2m\,c_{p}^{1-p}-\varepsilon\leq I_{p}(x)\leq 2m\,c_{p}^{1-p}+\varepsilon \,.
$$
Therefore, for every $t \geq t_\ep$, we get that
\begin{align}
8\pi m-4\pi c_{p}^{p-1}\,\varepsilon \, \leq \, F_{p}(t) \, \leq\,8\pi m+4\pi c_{p}^{p-1}\,\varepsilon\,,\nonumber
\end{align}
where we have used~\eqref{eq24} and~\eqref{eq23},
and in turn that 
$$
\lim_{t \to +\infty} F_p(t) \, = \, 8 \pi m \, .
$$
Combining this limit with the inequalities in~\eqref{eq21} 
and~\eqref{eq22} yields $2m \geq \beta_{p}$. 
By definition of $\beta_p$ and $c_p$, this means that
$$
2m \, \geq \,\Big(\frac{p-1}{3-p}\,\Big)^{\!\frac{p-1}{3-p}}\Bigg(\frac{\mathrm{Cap}_{p}(\partial M)}{4\pi}\Bigg)^{\!\!\frac{1}{3-p}}\,. 
$$
Finally, passing to the limit as $p\to 1^{+}$ with the help of~\cite[Theorem 1.2]{FM}, we get that
$$m\geq \sqrt{\frac{\vert\partial M^*\vert}{16\pi}}\,,$$
where $\pa M^*$ denotes the strictly outward minimising hull of $\pa M$
and is defined as in~\cite{FM}. 
{Assuming now that $\pa M$ is an outermost minimal surface, then} it is also strictly outward minimising and thus $|\pa M^*| = |\pa M|$, 
which implies the validity of the desired inequality.

\bigskip
\bigskip

\noindent
\textbf{{Acknowledgements.}} {\em 
The authors are members of the Gruppo Nazionale per l'Analisi Matematica, la Probabilit\`a e le loro Applicazioni (GNAMPA), which is part of the Istituto Nazionale di Alta Matematica (INdAM). Francesca Oronzio is supported by the ERC--STG Grant 759229: HiCoS ``Higher Co--dimension Singularities: Minimal Surfaces and the Thin Obstacle Problem''. The authors thank G.~Ascione, L.~Benatti and C.~Mantegazza for various useful discussions during the preparation of the manuscript. They also thank P.T.~Chru\'sciel, G.~Huisken for their interest on this work.}

\bibliographystyle{amsplain}

\providecommand{\bysame}{\leavevmode\hbox to3em{\hrulefill}\thinspace}
\providecommand{\MR}{\relax\ifhmode\unskip\space\fi MR }
\providecommand{\MRhref}[2]{%
  \href{http://www.ams.org/mathscinet-getitem?mr=#1}{#2}
}
\providecommand{\href}[2]{#2}

\end{document}